\documentclass{article}
\usepackage{amssymb}
\usepackage{amsmath}

\newcommand{\Eff}{{\cal E}\! f\! f}

\newcommand{\dar}{{\downarrow}}

\newtheorem{proposition}{Proposition}[section]
\newtheorem{lemma}[proposition]{Lemma}
\newtheorem{corollary}[proposition]{Corollary}
\newtheorem{definition}[proposition]{Definition}
\newtheorem{theorem}[proposition]{Theorem}
\newtheorem{exrcise}{Exercise}
\newtheorem{example}[proposition]{Example}
\newtheorem{remark}[proposition]{Remark}
\newcounter{opgaveteller}
\newcounter{lijst-teller}

\newcounter{boon}
\newcounter{boon2}
\newenvironment{proof}{\noindent {\bf Proof}. \nopagebreak }{\nopagebreak\hfill\rule{2mm}{3mm}}

\input xypic
\xyoption{all}
\UseComputerModernTips
\CompileMatrices
\definemorphism{dat}\dashed\tip\notip
\definemorphism{dub}\Solid\Tip\notip

\newenvironment{rlist}%
   {\begin{list}{\roman{lijst-teller}\/)\hfil}%
              {\labelwidth 2em%
               \leftmargin\labelwidth\advance\leftmargin by\labelsep%
               \usecounter{lijst-teller}}}%
   {\end{list}}

\newenvironment{r'list}%
   {\begin{list}{\roman{lijst-teller}\/)$'$\hfil}%
              {\labelwidth 2em%
               \leftmargin\labelwidth\advance\leftmargin by\labelsep%
               \usecounter{lijst-teller}}}%
   {\end{list}}

   {\end{list}}

\newenvironment{arlist}
    {\begin{list}{\arabic{boon2}\/)\hfil}%
                 {\labelwidth 2em%
                 \leftmargin\labelwidth\advance\leftmargin by\labelsep%
                 \usecounter{boon2}}}%
    {\end{list}}
\title{More on Geometric Morphisms between Realizability Toposes}
\author{Eric Faber \and Jaap van Oosten\footnote{Corresponding author. E-mail address: {\tt j.vanoosten@uu.nl}}\\
Department of Mathematics, Utrecht University\\
P.O.Box 80.010, 3508 TA Utrecht\\
The Netherlands}
\date{August 18, 2014}

\begin{document}

\maketitle
\begin{abstract}Geometric morphisms between realizability toposes are studied in terms of morphisms between partial combinatory algebras (pcas). The morphisms inducing geometric morphisms (the {\em computationally dense\/} ones) are seen to be the ones whose `lifts' to a kind of completion have right adjoints. We characterize topos inclusions corresponding to a general form of relative computability. We characterize pcas whose realizability topos admits a geometric morphism to the effective topos.
\end{abstract}

\section*{Introduction}
The study of geometric morphisms between realizability toposes was initiated by John Longley in his thesis \cite{LongleyJ:reatls}. Longley started an analysis of partial combinatory algebras (the structures underlying realizability toposes; see section~\ref{pcasection}) by defining a 2-categorical structure on them.

Longley's ``applicative morphisms'' characterize regular functors between categories of assemblies that commute with the global sections functors to Set. Longley was thus able to identify a class of geometric morphisms with adjunctions between partial combinatory algebras. The geometric morphisms thus characterized satisfy two constraints:\begin{arlist}
\item They are {\em regular}, that is: their direct image functors preserve regular epimorphisms.
\item They restrict to geometric morphisms between categories of assemblies.\end{arlist}
Restriction 1 was removed by Pieter Hofstra and the second author in \cite{HofstraP:ordpca}, where a new notion of applicative morphisms was defined, the {\em computationally dense\/} ones; these are exactly those applicative morphisms for which the induced regular functor on assemblies has a right adjoint (but the morphism itself need not have a right adjoint in the 2-category of partial combinatory algebras).

Restriction 2 was removed by Peter Johnstone in his recent paper \cite{JohnstonePT:geomrt}, where he proved that {\em every\/} geometric morphism between realizability toposes satisfies this condition.

Moreover, Johnstone gave a much simpler formulation of the notion of computational density.

In the present paper we characterize the computationally dense applicative morphisms in yet another way: as those which, when ``lifted'' to the level of {\em order-pcas}, do have a right adjoint. We also have a criterion for when the geometric morphism induced by a computationally dense applicative morphism is an inclusion.

In a short section we collect some material on total combinatory algebras, and formulate a criterion for when a partial combinatory algebra is isomorphic to a total one.

We prove that every realizability topos which is a subtopos of Hyland's {\em effective topos\/} is on a partial combinatory algebra of computations with an ``oracle'' for a partial function on the natural numbers. We employ a generalization of this ``computations with an oracle for $f$'' construction to arbitrary partial combinatory algebras, described in \cite{OostenJ:genfrr} and denoted $A[f]$. Generalizing results by Hyland (\cite{HylandJ:efft}) and Phoa (\cite{PhoaW:relcet}), we show that the inclusion of the realizability topos on $A[f]$ into the one on $A$ corresponds to the least local operator ``forcing $f$ to be realizable''.

The paper closes with some results about local operators in realizability toposes. We characterize the realizability toposes which admit a (necessarily essentially unique) geometric morphism to the effective topos, as those which have no De Morgan subtopos apart from Set.

In an effort to be self-contained, basic material is collected in section~\ref{intro}, which also establishes notation and terminology.

\section{Background}\label{intro}
\subsection{Partial Combinatory Algebras}\label{pcasection}
A {\em partial combinatory algebra\/} (or, as Johnstone calls them in \cite{JohnstonePT:glecrt,JohnstonePT:geomrt}, {\em Sch\"onfinkel algebra}) is a structure with a set $A$ and a partial binary function on it, which we denote by $a,b\mapsto ab$. This map is called {\em application}; the idea is that every element of $A$ encodes a partial function on $A$, and $ab$ is the result of the function encoded by $a$ applied to $b$.

The motivating example is the structure ${\cal K}_1$ on the set of natural numbers, where $ab$ is the outcome of the $a$-th Turing machine with input $b$.

Partial functions give rise to partial terms. In manipulating these we employ the following notational conventions:\begin{arlist}
\item The expression $t\dar$ means that the term $t$ is defined, or: denotes an element of $A$. We intend $t\dar$ to also imply that $s\dar$ for every subterm $s$ of $t$.
\item We employ association to the left: $abc$ means $(ab)c$. This economizes on brackets, but we shall be liberal with brackets wherever confusion is possible.
\item The expression $s\preceq t$ means: whenever $t$ denotes, so does $s$; and in that case, $s$ and $t$ denote the same element of $A$. We write $s\simeq t$ for the conjunction of $s\preceq t$ and $t\preceq s$. The expression $t=s$ means $s\simeq t$ and $t\dar$.\end{arlist}
With these conventions, we define:
\begin{definition}\label{pcadef}\em A set $A$ with a partial binary map on it is a {\em partial combinatory algebra\/} (pca) if there exist elements $\sf k$ and $\sf s$ in $A$ which satisfy, for all $a,b,c\in A$:\begin{rlist}
\item ${\sf k}ab=a$
\item ${\sf s}ab\dar$
\item ${\sf s}abc\preceq ac(bc)$\end{rlist}\end{definition}
This definition is mildly nonstandard, since most sources require $\simeq$ instead of $\preceq$ in clause iii). However, in our paper \cite{FaberE:effott} we show that in fact, every pca in our sense is isomorphic to a pca in the stronger sense (where the isomorphism is in the sense of applicative morphisms, see section~\ref{pcamorsection}), so the two definitions are essentially the same.

It is a consequence of definition~\ref{pcadef} that for any term $t$ which contains variables $x_1,\ldots ,x_{n+1}$, there is a term $\langle x_1\cdots x_{n+1}\rangle t$ without any variables, which has the following property: for all $a_1,\ldots a_{n+1}\in A$ we have\begin{itemize}
\item[] $(\langle x_1\cdots x_{n+1}\rangle t)a_1\cdots a_n\dar$
\item[] $(\langle x_1\cdots x_{n+1}\rangle t)a_1\cdots a_{n+1}\preceq t(a_1,\ldots ,a_{n+1})$\end{itemize}
Every pca $A$ has {\em pairing\/} and {\em unpairing\/} combinators: there are elements $\pi ,\pi _0,\pi _1$ of $A$ satisfying $\pi _0(\pi ab)=a$ and $\pi _1(\pi ab)=b$.

Moreover, every pca has {\em Booleans\/} {\sf T} and {\sf F} and a {\em definition by cases\/} operator: an element $u$ satisfying $u{\sf T}ab=a$ and $u{\sf F}ab=b$; such an element $u$ is seen as operating on three arguments $v,a,b$, which operation is often denoted by
$$\text{{\sf if} $v$ {\sf then} $a$ {\sf else} $b$}$$
In this paper we assume that ${\sf T}={\sf k}$ and ${\sf F}={\sf k}({\sf skk})$, so ${\sf T}ab=a$ and ${\sf F}ab=b$.

Finally, we mention that every pca $A$ comes equipped with a copy $\{\overline{n}\, |\, n\in\mathbb{N}\}$ of the natural numbers: the {\em Curry numerals}. For every $n$-ary partial computable function $F$, there is an element $a_F\in A$, such that for all $n$-tuples of natural numbers $k_1,\ldots k_n$ in the domain of $F$, $a_F\overline{k_1}\cdots\overline{k_n}=\overline{F(k_1,\ldots ,k_n)}$.
For more background on pcas we refer to \cite{OostenJ:reaics}, chapter 1.
\subsection{Assemblies}\label{asssection}
Every pca determines a category of {\em assemblies\/} on $A$, denoted ${\rm Ass}(A)$. An object of ${\rm Ass}(A)$ is a pair $(X,E)$ where $X$ is a set and $E$ associates to each element $x$ of $X$ a nonempty subset $E(x)$ of $A$. A morphism $(X,E)\to (Y,F)$ between assemblies on $A$ is a function $f:X\to Y$ of sets, for which there is an element $a\in A$ which {\em tracks\/} $f$, which means that for every $x\in X$ and every $b\in E(x)$, $ab\dar$ and $ab\in F(f(x))$.

The category ${\rm Ass}(A)$ has finite limits and colimits, is locally cartesian closed (hence regular), has a natural numbers object and a strong-subobject classifier (which is called a {\em weak subobject classifier\/} in \cite{JohnstonePT:skee}); hence it is a quasitopos.

There is an adjunction $\diagram {\rm Set}\rto<-0.5ex>_{\nabla} & {{\rm Ass}(A)}\lto<-0.5ex>_{\Gamma}\enddiagram$, $\Gamma\dashv\nabla$: here $\Gamma$ is the global sections functor (or the forgetful functor $(X,E)\mapsto X$) and $\nabla$ sends a set $X$ to the assembly $(X,E)$ where $E(x)=A$ for every $x\in X$.

The category ${\rm Ass}(A)$ is, except in the trivial case $A=1$, not {\em exact}. Its exact completion as a regular category (sometimes denoted ${\rm Ass}(A)_{\rm ex/reg}$) is a topos, the {\em realizability topos on\/} $A$, which we denote by ${\sf RT}(A)$ with only one exception: the topos ${\sf RT}({\cal K}_1)$ is called the {\em effective topos\/} and denoted $\Eff$. The effective topos was discovered by Martin Hyland around 1979 and described in the landmark paper \cite{HylandJ:efft}. The notation $\Eff$ serves both to underline the special place of the effective topos among realizability toposes (as we shall see in this paper) and the special place of ${\cal K}_1$ among pcas, and to acknowledge the seminal character of Hyland's work.
\subsection{Morphisms of Pcas}\label{pcamorsection}
In his thesis \cite{LongleyJ:reatls}, John Longley laid the groundwork for the study of the dynamics of pcas, by defining a useful 2-category structure on the class of pcas.
\begin{definition}\label{apmordef}\em Let $A$ and $B$ be pcas. An {\em applicative morphism\/} $A\to B$ is a total (or, as some people prefer, `entire') relation from $A$ to $B$, which we see as a map $\gamma$ from $A$ to the collection of nonempty subsets of $B$, which has a {\em realizer}, that is: an element $r\in B$ satisfying the following condition: whenever $a,a'\in A$ are such that $aa'\dar$, and $b\in\gamma (a),b'\in\gamma (a')$, then $rbb'\dar$ and $rbb'\in\gamma (aa')$.

Given two applicative morphisms $\gamma ,\delta :A\to B$ we say $\gamma\leq\delta$ if some element $s$ of $B$ satisfies: for every $a\in A$ and $b\in\gamma (a)$, $sb\dar$ and $sb\in\delta (a)$.\end{definition}
Pcas, applicative morphisms and inequalities between them form a preorder-enriched category. Applicative morphisms have both good mathematical properties and a computational intuition: if a pca is thought of as a model of computation, then an applicative morphism is a simulation of one model into another.

Mathematically, applicative morphisms correspond to `regular $\Gamma$-functors' between categories of assemblies: these are regular functors (functors preserving finite limits and regular epimorphisms) ${\rm Ass}(A)\to {\rm Ass}(B)$ which make the diagram
$$\diagram {{\rm Ass}(A)}\rrto\drto_{\Gamma} & & {{\rm Ass}(B)}\dlto^{\Gamma} \\
 & {\rm Set} & \enddiagram$$
commute. The construction is as follows: for an applicative morphism $\gamma :A\to B$, the functor $\gamma ^{\ast}$ sends the $A$-assembly $(X,E)$ to the $B$-assembly $(X,\gamma {\circ}E)$ ($\gamma {\circ}E$ is composition of relations). We have the following theorem:
\begin{theorem}[Longley]\label{longleythm} Every regular $\Gamma$-functor ${\rm Ass}(A)\to {\rm Ass}(B)$ is isomorphic to one of the form $\gamma ^{\ast}$ for an applicative morphism $\gamma :A\to B$. Moreover, there is a (necessarily unique) natural transformation $\gamma ^{\ast}\Rightarrow\delta ^{\ast}$, precisely when $\gamma\leq\delta$.\end{theorem}
Since ${\sf RT}(A)$ is the ex/reg completion of ${\rm Ass}(A)$, any functor of the form $\gamma ^{\ast}$ extends essentially uniquely to a regular functor ${\sf RT}(A)\to {\sf RT}(B)$, which we also denote by $\gamma ^{\ast}$. So it makes sense to study geometric morphisms ${\sf RT}(A)\to {\sf RT}(B)$ from the point of view of applicative morphisms $A\to B$: since the inverse image functor of any geometric morphism is regular, in order to study geometric morphisms ${\sf RT}(B)\to {\sf RT}(A)$ one looks at those applicative morphisms $\gamma :A\to B$ for which $\gamma ^{\ast}$ has a right adjoint.

The following definition is from \cite{HofstraP:ordpca}. Let us extend our notational conventions about application a bit: for $a\in A,\alpha\subseteq A$ we write $a\alpha\dar$ if $ax\dar$ for every $x\in\alpha$, and in this case we write $a\alpha$ for the set $\{ ax\, |\, x\in\alpha\}$.
\begin{definition}\label{cddef}\em An applicative morphism $\gamma :A\to B$ is {\em computationally dense\/} if there is an element $m\in B$ such that the following holds:\begin{itemize}
\item[] For every $b\in B$ there is an $a\in A$ such that for all $a'\in A$: if $b\gamma (a')\dar$, then $aa'\dar$ and $m\gamma (aa')\dar$ and $m\gamma (aa')\subseteq b\gamma (a')$.\end{itemize}\end{definition}
\begin{theorem}[\cite{HofstraP:ordpca}]\label{cdgmthm} An applicative morphism $\gamma :A\to B$ induces a geometric morphism ${\sf RT}(B)\to {\sf RT}(A)$ precisely when it is computationally dense.\end{theorem}
Obvious drawbacks of this theorem are the logical complexity of the definition of `computationally dense' and the fact that, prima facie, the theorem only says something about geometric morphisms which are induced by a $\Gamma$-functor between categories of assemblies, in other words: geometric morphisms ${\sf RT}(B)\to {\sf RT}(A)$ for which the inverse image functor maps assemblies to assemblies. Both these issues were successfully addressed in Peter Johnstone's paper \cite{JohnstonePT:geomrt}:
\begin{theorem}[Johnstone]\label{cdptj} An applicative morphism $\gamma :A\to B$ is computationally dense if and only if there exist an element $r\in B$ and a function $g:B\to A$ satisfying: for all $b\in B$ and all $b'\in\gamma (g(b))$, $rb'=b$. \end{theorem}
We might, extending the notation for inequalities between applicative morphisms, express the last property as: $\gamma g\leq {\rm id}_B$.
\begin{theorem}[Johnstone]\label{ptj}\begin{rlist}\item For any geometric morphism $f:{\sf RT}(B)\to {\sf RT}(A)$, the diagrams
$$\diagram {\rm Set}\dto\rto^{\rm id} & {\rm Set}\dto \\ {{\sf RT}(B)}\rto_f & {{\sf RT}(A)}\enddiagram$$
(where the vertical arrows embed Set as the category of $\neg\neg$-sheaves) is a bipullback in the 2-category of toposes and geometric morphisms.
\item For every geometric morphism $f:{\sf RT}(B)\to {\sf RT}(A)$, the inverse image functor $f^{\ast}$ preserves assemblies.\end{rlist}\end{theorem}
We shall be saying more about this theorem in section \ref{geommorsection}. For the moment, we continue out treatment of material from the literature, inasmuch it is relevant for our purposes. 
\begin{definition}\label{reggmdef}\em A geometric morphism is called {\em regular\/} if its direct image functor is a regular functor.\end{definition}
Clearly, by theorems \ref{longleythm} and \ref{ptj}, a regular geometric morphism ${\sf RT}(B)\to {\sf RT}(A)$ arises from an adjunction in the 2-category of pcas; and therefore Longley studied such adjunctions in his thesis. First, he distinguished a number of types of applicative morphisms:
\begin{definition}[Longley]\label{amtypes}\em Let $\gamma :A\to B$ be an applicative morphism.\begin{rlist}
\item $\gamma$ is called {\em decidable\/} if there is an element $d\in B$ such that for all $b\in\gamma ({\sf T}_A)$, $db={\sf T}_B$, and for all $b\in {\sf F}_A$, $db={\sf F}_B$.
\item $\gamma$ is called {\em discrete\/} if $\gamma (a)\cap\gamma (a')=\emptyset$ whenever $a\neq a'$.
\item $\gamma$ is called {\em projective\/} if $\gamma$ is isomorphic to an applicative morphism which is single-valued.\end{rlist}\end{definition}
Among other things, Longley proved the statements in the following theorem:
\begin{theorem}[Longley]\label{amtypesthm} Let $\diagram A\rto<-.5ex>_{\gamma} & B\lto<-.5ex>_{\delta}\enddiagram$ be a pair of applicative morphisms.\begin{rlist}
\item If $\gamma\delta\leq {\rm id}_B$ then $\gamma$ is decidable and $\delta $ is discrete.
\item If $\gamma\dashv\delta$ then $\gamma$ is projective.
\item If $\gamma\dashv\delta$ and $\delta\gamma\simeq {\rm id}_A$ then both $\delta$ and $\gamma$ are discrete and decidable.
\item $\gamma$ is decidable if and only if $\gamma ^{\ast}$ preserves finite sums, if and only if $\gamma ^{\ast}$ preserves the natural numbers object.
\item $\gamma$ is projective if and only if $\gamma ^{\ast}$ preserves regular projective objects.
\item $\gamma$ is discrete if and only if $\gamma ^{\ast}$ preserves discrete objects.
\item There exists, up to isomorphism, exactly one decidable applicative morphism ${\cal K}_1\to A$, for any pca $A$.\end{rlist}\end{theorem}
From theorem~\ref{amtypesthm} and theorem~\ref{ptj} we can draw some immediate inferences:
\begin{corollary}\label{amtypescor} Let $\gamma :A\to B$ be an applicative morphism.\begin{rlist}
\item If $\gamma$ is computationally dense, then $\gamma$ is decidable.
\item If $\gamma$ is computationally dense and the geometric morphism ${\sf RT}(B)\to {\sf RT}(A)$ induced by $\gamma$ is regular, then $\gamma$ is projective.
\item There exists, up to isomorphism, at most one geometric morphism ${\sf RT}(A)\to\Eff$; and there is one if and only if the essentially unique decidable morphism from ${\cal K}_1$ to $A$ is computationally dense.\end{rlist}\end{corollary}
We shall give an example where ii) fails, so not every geometric morphism is given by an adjunction on the level of pcas; and we shall give a criterion for iii) to hold, in terms of local operators on realizability toposes (theorem~\ref{dmtheorem}).

Let us draw one more corollary from theorem~\ref{amtypesthm}:
\begin{corollary}\label{regprojcor}Let $\gamma$ be computationally dense. Then the geometric morphism induced by $\gamma$ is regular, if and only if $\gamma$ has a right adjoint in {\rm PCA}, if and only if $\gamma$ is projective.\end{corollary}
\begin{proof} The first equivalence was already stated after definition~\ref{reggmdef}, and is a direct consequence of the biequivalence expressed by theorem~\ref{longleythm}. For the second equivalence, if $\gamma\dashv\delta$ then $\gamma$ is projective by \ref{amtypesthm}ii); conversely, if $\gamma$ is projective then by \ref{amtypesthm}v), the functor $\gamma ^{\ast}$ preserves regular projective objects, which, given that categories of assemblies always have enough regular projectives, is the case if and only if the right adjoint of $\gamma ^{\ast}$ preserves regular epimorphisms, and is therefore induced by some applicative morphism $\delta$, which by \ref{longleythm} must be right adjoint to $\gamma$ in PCA.\end{proof}
\subsection{Order-pcas}\label{opcasection}
Although most of our results are about ordinary pcas, the generalization to order-pcas, first defined in \cite{OostenJ:extrea} and elaborated on in \cite{HofstraP:ordpca}, has its advantages for the formulation of some results.
\begin{definition}\label{opcadef}\em An {\em order-pca\/} is a partially ordered set $A$ with a partial binary application function $(a,b)\mapsto ab$; there are also elements {\sf k} and {\sf s}, and the axioms are:\begin{rlist}
\item If $ab\dar$, $a'\leq a$ and $b'\leq b$ then $a'b'\dar$ and $a'b'\leq ab$
\item ${\sf k}ab\leq a$
\item ${\sf s}ab\dar$ and whenever $ac(bc)\dar$, ${\sf s}abc\dar$ and ${\sf s}abc\leq ac(bc)$\end{rlist}\end{definition}
\begin{definition}\label{opcaam}\em An {\em applicative morphism\/} of order-pcas $A\to B$ is a function $f:A\to B$ satisfying the following requirements:\begin{rlist}
\item There is an element $r\in B$ such that whenever $aa'\dar$ in $A$, $rf(a)f(a')\dar$ in $B$, and $rf(a)f(a')\leq f(aa')$.
\item There is an element $u\in B$ such that whenever $a\leq a'$ in $A$, $uf(a)\dar$ and $uf(a)\leq f(a')$ in $B$.\end{rlist}\end{definition}
Just as for pcas, we have an order on applicative morphisms, which is analogously defined.

Every order-pca $A$ determines a category of assemblies: objects are pairs $(X,E)$ where $X$ is a set and $E(x)$ is a nonempty, downward closed subset of $A$, for each $x\in X$; morphisms are set-theoretic functions which are tracked just as in the definition for pcas.

On the 2-category of order-pcas there is a 2-monad $T$, which at the same time gives the prime examples of interest of genuine order-pcas: $T(A)$ is the order-pca consisting of nonempty, downward closed subsets of $A$, with the inclusion ordering; for $\alpha ,\beta\in T(A)$, we say $\alpha\beta\dar$ if and only if for all $a\in\alpha$ and $b\in\beta$, $ab\dar$ in $A$; if that holds, $\alpha\beta$ is the downward closure of the set $\{ ab\, |\, a\in\alpha ,b\in\beta\}$.

Note that when we consider applicative morphisms $f$ to order-pcas of the form $T(A)$, we may assume that $f$ is an order-preserving function; since the element $u$ of \ref{opcaam}ii) allows us to find an isomorphism between $f$ and the map $x\mapsto\bigcup_{y\leq x}f(y)$.

The category of assemblies on the order-pca $T(A)$ has enough regular projectives: a $T(A)$-assembly $(X,E)$ is regular projective if and only if (up to isomorphism) $E(x)$ is a {\em principal downset\/} of $T(A)$ for each $x$; i.e., $E(x)=\{\alpha\subseteq A\, |\,\alpha\subseteq\beta\}$ for some $\beta\in T(A)$. It is now easy to see that the full subcategory of ${\rm Ass}(T(A))$ on the regular projectives is equivalent to ${\rm Ass}(A)$, and applying a criterion due to Carboni (\cite{CarboniA:somfcr}), one readily verifies
\begin{theorem}\label{TAregcomplthm}The category of assemblies on $T(A)$ is the regular completion of the category ${\rm Ass}(A)$.\end{theorem}
\subsection{Relative recursion}\label{relrecsection}
We also need to recall a construction given in \cite{OostenJ:genfrr}. Given a pca $A$ and a partial function $f:A\to A$, we say that $f$ is {\em representable\/} w.r.t.\ an applicative morphism $\gamma :A\to B$, if there is an element $b\in B$ which satisfies: for each $a$ in the domain of $f$ and each $c\in\gamma (a)$, $bc\dar$ and $bc\in\gamma (f(a))$. We say that $f$ is {\em representable}, or {\em representable in $A$}, if $f$ is representable w.r.t.\ the identity morphism on $A$.

There is a pca $A[f]$ and a decidable applicative morphism $\iota _f:A\to A[f]$ such that $f$ is representable w.r.t.\ $\iota _f$ and $\iota _f$ is universal with this property: whenever $\gamma :A\to B$ is a decidable applicative morphism w.r.t.\ which $f$ is representable, then $\gamma$ factors uniquely through $\iota _f$.

It follows that this property determines $A[f]$ up to isomorphism, and hence, if $f$ is representable in $A$ then $A$ and $A[f]$ are isomorphic.

Moreover, the applicative morphism $\iota _f$ is computationally dense and induces an inclusion of toposes: ${\sf RT}(A[f])\to {\sf RT}(A)$. Moreover, $\iota _f$, being the identity function on the level of sets, is projective as applicative morphism. 
\section{Geometric morphisms between realizability toposes}\label{geommorsection}
We start by formulating a variation on Longley's theorem~\ref{longleythm}. Recall the definition of order-pcas and the monad $T$ from section~\ref{opcasection}. We wish to characterize finite limit-preserving $\Gamma$-functors between categories of assemblies.
\begin{definition}\label{protoam}\em Let $A,B$ be pcas. A {\em proto-applicative morphism\/} from $A$ to $B$ is an applicative morphism of order-pcas from $T(A)$ to $T(B)$.\end{definition}
\begin{theorem}\label{assflpthm} There is a biequivalence between the following two 2-categories:\begin{itemize}
\item[1] The category of pcas, proto-applicative morphisms and inequalities between them
\item[2] The category of categories of the form ${\rm Ass}(A)$ for a pca $A$, finite limit-preserving $\Gamma$-functors and natural transformations\end{itemize}\end{theorem}
\begin{proof} Let $\gamma :T(A)\to T(B)$ be an applicative morphism, realized by $r\in B$. Define $\gamma ^{\ast}(X,E)=(X,\gamma\circ E)$. If $f:(X,E)\to (Y,E')$ is tracked by $t\in A$, then $$r\gamma (\{ t\} )\gamma (E(x))\subseteq\gamma (E'(f(x)))$$
so whenever $s\in\gamma (\{ t\} )$, $rs$ tracks $f$ as morphism $(X,\gamma\circ E)\to (Y,\gamma\circ E')$. So $\gamma ^{\ast}$ is a $\Gamma$-functor.

It is immediate that $\gamma ^{\ast}$ preserves terminal objects and equalizers; that $\gamma ^{\ast}$ preserves finite products is similar to the proof of theorem~\ref{longleythm} (for which the reader may consult either \cite{LongleyJ:reatls} or  \cite{OostenJ:reaics}.

If $\gamma\leq\delta :T(A)\to T(B)$ is realized by $\beta\in T(B)$ and $b\in\beta$, then $b$ tracks every component of the unique natural transformation $\gamma ^{\ast}\Rightarrow\delta ^{\ast}$. Conversely, suppose there is a natural transformation $\gamma ^{\ast}\Rightarrow\delta ^{\ast}$, consider its component at the object $(T(A),i)$ where $i$ is the identity function. Any element of $B$ which tracks this component realizes $\gamma\leq\delta$.

Now suppose that $F:{\rm Ass}(A)\to {\rm Ass}(B)$ is a finite-limit preserving $\Gamma$-functor. We may well suppose that $F$ is the identity on the level of sets, as any $\Gamma$-functor is isomorphic to a functor having this property. Consider again the object $(T(A), i)$ of ${\rm Ass}(A)$ and its $F$-image $(T(A),\tilde{F})$ in ${\rm Ass}(B)$, for some map $\tilde{F}:T(A)\to T(B)$. We wish to show that $\tilde{F}$ is a proto-applicative morphism $A\to B$.

Let $P\, =\, \{ (\alpha ,\beta )\in T(A)\times T(A)\, |\,\alpha\beta\dar\}$. For $(\alpha ,\beta )\in P$ put $E(\alpha ,\beta )=\pi\alpha\beta$ (where $\pi$ is the pairing combinator in $A$). Then $(P,E)$ is a regular subobject of $(T(A),i)\times (T(A),i)$ in ${\rm Ass}(A)$ so by assumption on $F$, $F(P,E)$ is a regular subobject of $(T(A),\tilde{F})\times (T(A),\tilde{F})$; we may assume that $F(P,E)=(P,\hat{E})$ with $\hat{E}(\alpha ,\beta )=\rho\tilde{F}(\alpha )\tilde{F}(\beta )$ (where $\rho$ is the pairing combinator in $B$). There is an application map ${\rm app}:(P,E)\to (T(A),i)$, hence we have a map ${\rm app}:(P,\hat{E})\to (T(A),\tilde{F})$. Modulo a little fiddling with realizers, any element of $B$ tracking this map realizes $\tilde{F}$ as applicative morphism $T(A)\to T(B)$.

Furthermore, since any natural transformation between the sort of functors we consider is the identity on the level of sets, if we have a natural transformation $F\Rightarrow G$ then we have a tracking for the identity function as morphism $(T(A),\tilde{F})\to (T(A),\tilde{G})$; such a tracking realizes $\tilde{F}\leq\tilde{G}$.

It is immediate that $\tilde{\gamma ^{\ast}}=\gamma$. The proof that $(\tilde{F})^{\ast}\simeq F$ is similar to the proof of the analogous statement in Longley's theorem.\end{proof}
\medskip

\noindent We can now give another characterization of computationally dense applicative morphisms of pcas. Every applicative morphism $\gamma :A\to B$ of pcas is also an applicative morphism $A\to T(B)$ of order-pcas and hence induces an applicative morphism $\tilde{\gamma}:T(A)\to T(B)$ (and the functors $\gamma ^{\ast}$ from \ref{longleythm} and $(\tilde{\gamma})^{\ast}$ of \ref{assflpthm} coincide); by the biequivalence in the latter theorem, we have the following corollary:
\begin{corollary}\label{cdcharthm} For an applicative morphism $\gamma :A\to B$ the following statements are equivalent:\begin{rlist}
\item $\gamma$ is computationally dense
\item $\tilde{\gamma}$ has a right adjoint (in the 2-category of order-pcas)
\item there is an applicative morphism $\delta :B\to A$ such that $\gamma\delta\leq {\rm id}_B$\end{rlist}\end{corollary}
\begin{proof} i)$\Rightarrow$ii): if $\gamma$ is computationally dense then it induces a geometric morphism ${\sf RT}(B)\to {\sf RT}(A)$ which, by \ref{ptj}, restricts to an adjunction between $\Gamma$-functors on the categories of assemblies; by \ref{assflpthm} this is induced by an adjunction between proto-applicative morphisms.

\noindent ii)$\Rightarrow$iii): let $\delta :T(B)\to T(A)$ be right adjoint to $\tilde{\gamma}$. Define $\bar{\delta}:B\to T(A)$ by
$$\bar{\delta}(b)\; =\; \delta (\{ b\} )$$
Then $\bar{\delta}$ is an applicative morphism $A\to B$ and $\gamma\bar{\delta}\leq {\rm id}_B$ since $\gamma\bar{\delta}(b)=\tilde\gamma\delta (\{ b\} )$ and $\tilde{\gamma}\dashv\delta$.

\noindent iii)$\Rightarrow$i): this is immediate from \ref{ptj} \end{proof}  
\medskip

\noindent Another corollary is the following: 
\begin{corollary}\label{cdcharcor}The following data are equivalent:\begin{rlist}
\item a geometric morphism ${\sf RT}(B)\to {\sf RT}(A)$ 
\item an adjunction $\diagram {{\rm Ass}(B)}\rto<-.5ex>_{f_{\ast}} & {{\rm Ass}(A)}\lto<-.5ex>_{f^{\ast}}\enddiagram$, $f^{\ast}\dashv f_{\ast}$, and $f^{\ast}$ preserving finite limits
\item an adjunction $\diagram {T(B)}\rto<-.5ex>_{\gamma _{\ast}} & {T(A)}\lto<-.5ex>_{\gamma ^{\ast}}\enddiagram$, $\gamma ^{\ast}\vdash\gamma _{\ast}$, in the 2-category of order-pcas
\item a computationally dense applicative morphism $A\to B$\end{rlist}\end{corollary}
\begin{proof} By \ref{cdgmthm} and \ref{ptj}, i) and iv) are equivalent and imply ii); the equivalence between ii) and iii) is theorem~\ref{assflpthm}. Suppose we have an adjunction as in ii). Then $f_{\ast}$ is always a $\Gamma$-functor, since $\Gamma$ is represented by 1 and $f^{\ast}$ preserves 1. So $f_{\ast}$ is, by \ref{assflpthm}, induced by a proto-applicative morphism; but such functors always commute with $\nabla$ (alternatively, one may apply a -- non-constructive -- theorem, 2.3.3 from \cite{LongleyJ:reatls}, which tells us that {\em every\/} $\Gamma$-functor between categories of assemblies commutes with $\nabla$) and therefore their left adjoints commute with $\Gamma$ and we have an adjunction of $\Gamma$-functors, hence an adjunction of proto-applicative morphisms, hence a computationally dense morphism $A\to B$.\end{proof}
\medskip

\noindent In the same way we can characterise which computationally dense $\gamma :A\to B$ induce geometric {\em inclusions}:
\begin{corollary}\label{incchar} A computationally dense applicative morphism $\gamma :A\to B$ induces an inclusion of toposes: ${\sf RT}(B)\to {\sf RT}(A)$ if and only if there is an applicative morphism $\delta :B\to A$ such that $\gamma\delta\simeq {\rm id}_B$.\end{corollary}
We conclude this section with the promised example of a computationally dense applicative morphism which is not projective:
\begin{example}\em Consider the pca ${\cal K}_2^{\rm rec}$ (see \cite{OostenJ:reaics}, 1.4.9) and the applicative morphism ${\cal K}_2^{\rm rec}\to {\cal K}_1$ which sends every total recursive function to the set of its indices (\cite{OostenJ:reaics}, p.\ 95). For recursion-theoretic reasons, this can {\em not\/} be isomorphic to a single-valued relation, so this is an example of a geometric morphism $\Eff\to {\sf RT}({\cal K}_2^{\rm rec})$ which is not regular.\end{example}

\subsection{Intermezzo: total pcas}
In this small section we include some material on total pcas; it contains a characterization of the pcas which are isomorphic to a total one.

A pca $A$ is called {\em total\/} if for all $a$ and $b$, $ab\dar$. The following results have been established about total pcas:\begin{itemize}
\item The topos $\Eff$ is not equivalent to a realizability topos on a total pca (\cite{JohnstonePT:irt}).
\item Every total pca is isomorphic to a nontotal one (\cite{OostenJ:genfrr}).
\item Every realizability topos is covered (in the sense of a geometric surjection) by a realizability topos on a total pca (\cite{OostenJ:parcaf}).
\end{itemize}

\begin{definition}\label{almosttotal}\em Call an element $a$ of a pca $A$ {\em total\/} if for all $b\in A$, $ab\dar$. Call a pca $A$ {\em almost total\/} if for every $a\in A$ there is a total element $b\in A$ such that for all $c\in A$, $bc\preceq ac$.

A pca is called {\em decidable\/} if there is an element $d\in A$ which decides equality in $A$, that is: for all $a,b\in A$,
$$dab\; =\; \left\{\begin{array}{cl} {\sf T} & \text{if }a=b\\ {\sf F} & \text{if }a\neq b\end{array}\right.$$\end{definition}
\begin{proposition}\label{decnottotal} A nontrivial decidable pca is never almost total.\end{proposition}
\begin{proof} Let $A$ be nontrivial and decidable. Choose $e\in A$ such that for all $x\in A$, $ex\simeq x{\sf k}$. Pick elements $a\neq b\in A$. Suppose that $g$ is a total element for $e$ as in definition~\ref{almosttotal}. By the recursion theorem for $A$ (\cite{OostenJ:reaics}, 1.3.4) there is $h\in A$ satisfying for all $y\in A$:
$$hy\simeq d(gh)aba$$
Then $hy=b$ if $gh=a$, and $hy=a$ otherwise (recall that ${\sf T}xy=x,{\sf F}xy=y$). Since $h$ is total, we have $eh=h{\sf k}$. But now,
$$eh\; =\; h{\sf k}\; =\;\left\{\begin{array}{cl}b & \text{if }gh=a \\ a & \text{if }gh\neq a\end{array}\right.\; =\;\left\{\begin{array}{cl}b & \text{if }eh=a\\ a & \text{if }eh\neq a\end{array}\right.$$
A clear contradiction.\end{proof}
\begin{proposition}\label{altotalprop}Let $A$ be a pca. The following four conditions are equivalent:\begin{rlist}
\item $A$ is almost total.
\item There is an element $g\in A$ such that for all $e\in A$, $ge$ is total and for all $x$, $gex\preceq ex$.
\item $A$ is isomorphic to a total pca.\end{rlist}\end{proposition}
\begin{proof} i)$\Rightarrow$ii): assume $A$ is almost total. Pick $f\in A$ such that for all $y$, $fy\simeq\pi _0y(\pi _1y)$ (recall that $\pi ,\pi _0,\pi _1$ are the pairing and unpairing combinators in $A$). By assumption there is a total element $h$ for $f$ as in definition~\ref{almosttotal}. Let $g$ be such that $gxy\simeq h(\pi xy)$.

Then for every $e\in A$, $ge$ is a total element and if $ex\dar$ then
$$gex=h(\pi ex)=f(\pi ex)=ex$$
so $gex\preceq ex$ as required.

\noindent ii)$\Rightarrow$iii): assume $A$ satisfies condition ii). Define a binary function $\ast$ on $A$ by putting $a\ast b=gab$. We have
$${\sf k}\ast a\ast b=g(g{\sf k}a)b={\sf k}ab=a$$
and if ${\sf s}'=\langle xyz\rangle g(gxz)(gyz)$ then
$$\begin{array}{lclcl} {\sf s}'\ast a\ast b\ast c & = & g(g(g{\sf s}'a)b)c & = & g(\langle z\rangle g(gaz)(gbz))c  \\
 & = & g(gac)gbc) & = & a\ast c\ast (b\ast c)  \end{array}$$
 So, $(A,\ast )$ is a total pca. The identity function $A\to A$ is an applicative morphism $A\to (A,\ast )$, realized by ${\sf s}'\ast {\sf k}\ast {\sf k}$ in $(A,\ast )$, and in the other direction it is realized by $g\in A$. So $A$ is isomorphic to $(A,\ast )$.
 
\noindent iii)$\Rightarrow$i): suppose $A$ is isomorphic to $B$ and $B$ is total. By \ref{amtypesthm}ii) we may assume that the isomorphism is given by functions $f:A\to B$ and $g:B\to A$ which are each other's inverse; suppose $r\in B$ realizes $f$ as applicative morphism, and $s\in A$ realizes $g$.

For $a\in A$ let $a'=s(sg(r)gf(a))$. For any $x\in A$ we have:
$$\begin{array}{lclcl} a'x & = & s(sg(r)gf(a))x & = & s(g(rf(a))x \\
 & = & s(g(rf(a)))gf(x) & = & g(rf(a)f(x) \end{array}$$
So, $a'x\dar$, and if $ax\dar$ then $a'x=gf(ax)=ax$. So $A$ is almost total, as desired.\end{proof}

\subsection{Discrete computationally dense morphisms}
We employ the following convention for a parallel pair of geometric morphisms $\alpha ,\beta$ between realizability toposes: we write $\alpha\leq\beta$ if there is a (necessarily unique) natural transformation $\alpha ^{\ast}\Rightarrow\beta ^{\ast}$.
\begin{theorem}\label{discthm}Let $\gamma :A\to B$ be a discrete, computationally dense applicative morphism.\begin{rlist}
\item There is a pca of the form $A[f]$ such that the geometric morphism ${\sf RT}(B)\to {\sf RT}(A)$ factors through the inclusion ${\sf RT}(A[f])\to {\sf RT}(A)$ by a geometric morphism $\alpha :{\sf RT}(B)\to {\sf RT}(A[f])$
\item Moreover, there is a geometric morphism $\beta :{\sf RT}(A[f])\to {\sf RT}(B)$ satisfying $\alpha\beta\leq {\rm id}_{{\sf RT}(A[f])}$ and ${\rm id}_{{\sf RT}(B)}\leq\beta\alpha$.
\item If $\gamma$ induces an inclusion of toposes, then $\beta\alpha\simeq {\rm id}_{{\sf RT}(B)}$, so ${\sf RT}(B)$ is a retract of ${\sf RT}(A[f])$.
\item If $\gamma$ is projective then $\alpha\beta\simeq {\rm id}_{{\sf RT}(A[f])}$, so ${\sf RT}(A[f])$ is a retract of ${\sf RT}(B)$.
\item Hence, if $\gamma$ is projective and induces an inclusion, ${\sf RT}(B)$ is equivalent to ${\sf RT}(A[f])$.\end{rlist}\end{theorem}
\begin{proof} By \ref{cdcharthm}ii), $\tilde{\gamma}:T(A)\to T(B)$ has a right adjoint $\delta$. Let us write $\delta '$ for the morphism $\bar{\delta}$ from the proof of \ref{cdcharthm}: $\delta '(b)=\delta (\{ b\} )$. Assume, as we may, that $\delta $ preserves inlusions. This means that $\tilde{\delta '}\leq\delta$ as morphisms $T(B)\to T(A)$. We have $\gamma\delta '\leq {\rm id}B$, so $\delta '$ is discrete by \ref{amtypesthm}i).

Since both $\gamma$ and $\delta '$ are discrete, so is $\delta '\gamma$ and we have a partial function $f:A\to A$ defined by: $f(a)=b$ if and only if $a\in\delta '\gamma (b)$. The partial function $f$ is representable w.r.t.\ $\gamma$, for if $\epsilon\in B$ realizes $\gamma\delta '\leq {\rm id}_B$ and $f(a)=b$, $c\in\gamma (a)$, then $a\in\delta '\gamma (b)$ so $c\in\gamma\delta '\gamma (b)$ so $\epsilon c\in\gamma (b)$; hence $\epsilon $ represents $f$ w.r.t.\ $\gamma$. Since, by section~\ref{relrecsection}, $\gamma$ factors through $\iota _f$, the geometric morphism ${\sf RT}(B)\to {\sf RT}(A)$ factors through the inclusion ${\sf RT}(A[f])\to {\sf RT}(A)$. This proves i). The geometric morphism $\alpha$ is induced by $(\gamma _f)^{\ast}\dashv\delta _f$ on the level of assemblies; here $\gamma _f$ and $\delta _f$ are the same relations on the level of sets as $\gamma ,\delta$ respectively.
\medskip

\noindent We can regard $\delta '$ also as applicative morphism $B\to A[f]$. Now in $A[f]$, $\delta '\gamma _f\leq {\rm id}_{A[f]}$, since if $u$ represents $f$ in $A[f]$ then $u$ realizes this inequality. This means that $\delta ':B\to A[f]$ is computationally dense, by \ref{cdcharthm}. So there is a geometric morphism $\beta :{\sf RT}(A[f])\to {\sf RT}(B)$; let $\zeta :T(A[f])\to T(B)$ be the right adjoint to $\tilde{\delta '}$. We have the following diagram of order-pcas:
$$\diagram {T(B)}\rto<-.5ex>_{\delta} & {T(A[f])}\lto<-.5ex>_{\tilde{\gamma}}\rto<-.5ex>_{\zeta} & {T(B)}\lto<-.5ex>_{\tilde{\delta '}}\enddiagram$$
We have $\tilde{\gamma}\tilde{\delta '}\leq\tilde{\gamma}\delta\leq {\rm id}_{T(B)}$ and $\tilde{\delta '}\tilde{\gamma}\leq {\rm id}_{T(A[f])}$, so this proves ii).
\medskip

\noindent If $\gamma$ induces an inclusion then $\gamma\delta '\simeq {\rm id}_B$ so $\beta\alpha\simeq {\rm id}_{{\sf RT}(B)}$ and ${\sf RT}(B)$ is a retract of ${\sf RT}(A[f])$.\
\medskip

\noindent If $\gamma$ is projective then $\tilde{\delta '}\simeq\delta$ so $\delta\tilde{\gamma}\simeq\tilde{\delta '}\tilde{\gamma}\leq {\rm id}_{T(A[f])}\leq\delta\tilde{\gamma}$, so $\alpha\beta$ is isomorphic to the identity on ${\sf RT}(A[f])$ and this topos is a retract of ${\sf RT}(B)$.
\medskip

\noindent v) is obvious.\end{proof}
\begin{corollary}If $A$ is a decidable pca, then every realizability topos which is a subtopos of ${\sf RT}(A)$ is a retract of ${\sf RT}(A[f])$ for some partial function $f:A\to A$.

Every realizability topos which is a subtopos of $\Eff$ is equivalent to one of the form ${\sf RT}({\cal K}_1[f])$ for some partial function on the natural numbers.\end{corollary}
\begin{proof} Both statements follow from theorem~\ref{discthm}, sinca if $A$ is decidable, then for every computationally dense applicative morphism $\gamma :A\to B$ we have that $\gamma ^{\ast}(A,\{\cdot\} )$ is decidable in ${\rm Ass}(B)$, hence discrete; and therefore $\gamma$ is discrete. For the second statement, note that the essentially unique decidable applicative morphism ${\cal K}_1\to B$ is discrete and projective.\end{proof}

\section{Local Operators in Realizability Toposes}
Local operators ($j$-operators, Lawvere-Tierney topologies) in the Effective topos have been studied in \cite{HylandJ:efft,PittsAM:thet,LeeS:basse,OostenJ:realop}. We quickly recall some basic facts which readily generalize to arbitrary realizability toposes.

Let $A$ be a pca. For subsets $U,V$ of $A$ we denote by $U\Rightarrow V$ the set of all elements $a\in A$ which satisfy: for every $x\in U$, $ax\dar$ and $ax\in V$. We write $U\wedge V$ for the set $\{ \pi ab\, |\, a\in U,b\in V\}$. The powerset of $A$ is denoted ${\cal P}(A)$.

Every local operator in ${\sf RT}(A)$ is represented by a function $J: {\cal P}(A)\to {\cal P}(A)$ for which the sets\begin{rlist}
\item $\bigcap_{U\subseteq A}U\Rightarrow J(U)$
\item $\bigcap_{U\subseteq A}JJ(U)\Rightarrow J(U)$
\item $\bigcap_{U,V\subseteq A}(U\Rightarrow V)\Rightarrow (J(U)\Rightarrow J(V))$\end{rlist}
are all nonempty. A map $J$ for which just the set iii) is nonempty, is said to represent a {\em monotone map\/} on $\Omega$. Abusing language, we shall just speak of ``local operators'' and ``monotone maps'' when we mean the maps representing them.

\begin{example}\label{locopexamples}\em Important examples of local operators are:\begin{arlist}
\item The identity map on ${\cal P}(A)$; this is the least local operator, and denoted by $J_{\bot}$. Its category of sheaves is just ${\sf RT}(A)$ itself.
\item The constant map with value $A$. This is the largest local operator, denoted $J_{\top}$; its category of sheaves is the trivial topos.
\item The map which sends every nonempty set to $A$, and the empty set to itself. This is the $\neg\neg$-operator, and we shall also denote it by $\neg\neg$. Its category of sheaves is Set.
\item Suppose $\gamma :A\to B$ is a computationally dense applicative morphism, inducing $\tilde{\gamma}:T(A)\to T(B)$ and its right adjoint $\delta$ by the theory of section~\ref{geommorsection}. The map $J:{\cal P}(A)\to {\cal P}(A)$ which sends the empty set to itself, and every nonempty $U\subseteq A$ to $\delta\tilde{\gamma}(U)$, is a local operator; its category of sheaves is the image of the geometric morphism ${\sf RT}(B)\to {\sf RT}(A)$ induced by $\gamma$.
\end{arlist}\end{example}

There is a partial order on local operators: $J\leq K$ iff the set $$\bigcap_{U\subseteq A}J(U)\Rightarrow K(U)$$ is nonempty (strictly speaking this gives a preorder on representatives of local operators). Every local operator is represented by a map $J$ which preserves inclusions (\cite{LeeS:basse}, Remark 2.1). If $M:{\cal P}(A)\to {\cal P}(A)$ is a monotone map, there is a least local operator $J_M$ such that $M\leq J_M$: it is given by
$$J_M(U)\; =\; \bigcap\{ Q\subseteq A\, |\, \{ {\sf T}\}\wedge U\subseteq Q\text{ and } \{ {\sf F}\}\wedge M(Q)\subseteq Q\}$$

It is a general fact of topos theory that for any monomorphism $m$ in a topos there is a least local operator which ``inverts $m$'', i.e.\ for which the sheafification of $m$ is an isomorphism. In ${\sf RT}(A)$, every object is covered by an $A$-assembly, so we need only consider monos into assemblies. Here, we restrict ourselves to two types of monos:\begin{enumerate}
\item Consider an assembly $(X,E)$ and the mono $(X,E)\to\nabla (X)$. Let $M$ be the monotone map sending $U\subseteq A$ to the set
$$\bigcup_{x\in X}E(x)\Rightarrow U$$
Then $J_M$ is the least local operator inverting the mono $(X,E)\to\nabla (X)$.
\item Consider a partial function $f:A\to A$ with domain $B\subseteq A$. We have the assemblies $(B,\{\cdot\} )$ and $(B,E)$ where $E(b)=\{\pi bf(b)\}$. The identity on $B$ is a map of assemblies $(B,E)\to (B,\{\cdot\} )$, tracked by $\pi _0$. The least local operator inverting this mono (``forcing $f$ to be realizable'') is $J_M$, where $M$ is the monotone map
$$U\;\mapsto\;\{\pi be\, |\, ef(b)\in U\}$$\end{enumerate}
The following theorem generalizes a result by Hyland and Phoa (\cite{HylandJ:efft,PhoaW:relcet}). 
\begin{theorem}\label{relrecA}The category of sheaves for the local operator of type 2 above, is ${\sf RT}(A[f])$.\end{theorem}
\begin{proof} We refer to \cite{OostenJ:genfrr} for details on $A[f]$. The underlying set of $A[f]$ is $A$; the application map of $A[f]$ is denoted $a,b\mapsto a{\cdot}^fb$.

It follows from the construction of the elements {\sf k} and {\sf s} in $A[f]$, that if $t(x,a_1,\ldots ,a_n)$ is a term built from variable $x$, parameters $a_1,\ldots ,a_n\in A$ and the application of $A[f]$, that the element $\langle x\rangle t(x,\vec{a})$ of $A[f]$ can be obtained computably in $A$ from the parameters $\vec{a}$.

The computationally dense applicative morphism $\iota _f:A\to A[f]$ is just the identity function, and the right adjoint $\delta :T(A[f])\to T(A)$ is given by
$$\delta (U)\; =\; \{\pi ae\, |\, e{\cdot}^fa\in U\}$$
Indeed, ${\rm id}_{T(A)}\leq\delta\tilde{\iota _f}=\delta$ because we can find, $A$-computably in $a$, an element $\xi _a$ satisfying $\xi _a{\cdot}^fx=a$ for all $x$. Also, $\delta =\tilde{\iota _f}\delta\leq {\rm id}_{T(A[f])}$ by simply evaluating in $A[f]$. We need to see that $\delta$ is applicative; but if $U{\cdot}^fV$ is defined in $T(A[f])$ and $\pi ae\in\delta (U)$, $\pi bc\in\delta (V)$ then
$$\pi (\pi ab)(\langle x\rangle (e{\cdot}^f(\pi _0x)){\cdot}^f(c{\cdot}^f(\pi _1x)))$$
is an element of $\delta (U{\cdot}^fV)$ and we noted that this element can be obtained $A$-computably from $a,e,b,c$.
\medskip

So, by Example~\ref{locopexamples}, item 4, the local operator on ${\sf RT}(A)$ for which the category of sheaves is ${\sf RT}(A[f])$, sends $U$ to $\{\pi ae\, |\, e{\cdot}^fa\in U\}$. Let us call this map $J_f$.

On the other hand, the least local operator which forces the partial function $f$ to be realizable, is the map $J_M$ where $M$ is the monotone map
$$U\,\mapsto\,\{\pi be\, |\, ef(b)\in U\}$$
We need to prove $J_M\leq J_f$ and $J_f\leq J_M$.

By $[a_1,\ldots ,a_n]$ we denote some standard coding in $A$ of the $n$-tuple $a_1,\ldots ,a_n$. We write $[\, ]$ for the code of the empty tuple. The symbol $\ast$ is used for ($A$-computable) concatenation of tuples: so $$[a_1,\ldots ,a_n]\ast [b_1,\ldots ,b_m]=[a_1,\ldots ,a_n,b_1,\ldots ,b_m]$$

The definition of $a{\cdot}^fb=c$ is as follows:\begin{itemize}
\item[] $a{\cdot}^fb=c$ if and only if either $a[b]=\pi {\sf T}c$, or there is a sequence $a_1,\ldots ,a_n$ such that $a[b]=\pi {\sf F}d$ for some $d$ such that $f(d)=a_1$, and for all $1<k\leq n$, $a[b,a_1,\ldots ,a_{k-1}]=\pi {\sf F}d$ for some $d$ such that $f(d)=a_k$, and moreover, $a[b,a_1,\ldots ,a_n]=\pi {\sf T}c$
\end{itemize}
Let us call such a sequence $a_1,\ldots ,a_n$ a {\em computation sequence\/} for $a{\cdot}^fb$.

Now clearly, if $a\in A$ satisfies $ae[b]=\pi {\sf F}b$ and $ae[b,c]\simeq ec$ for all $e,b,c$, then we have $ae{\cdot}^fb\simeq ef(b)$. Hence, $M\leq J_f$ and therefore $J_M\leq J_f$ by definition of $J_M$.

For the converse, let $\alpha\in\bigcap_{U\subseteq A}U\Rightarrow J_M(U)$ and $\zeta\in\bigcap _{U\subseteq A}M(J_M(U))\Rightarrow J_M(U)$. By the recursion theorem in $A$, there is an element $\gamma\in A$ such that for all $e,a,\sigma$:
$$\begin{array}{lll}\gamma ea\sigma\ & \preceq & {\sf If}\;\; \pi _0(e([a]\ast\sigma )) \\
 & & {\sf then}\;\; \alpha (\pi _1(e([a]\ast\sigma ))) \\
 & & {\sf else}\;\;\zeta (\pi (\pi _1(e([a]\ast\sigma )))\langle x\rangle\gamma ea(\sigma\ast [x])) \end{array}$$
We claim: if $e{\cdot}^fa\in U$, and $a_1,\ldots ,a_n$ is a computation sequence for $e{\cdot}^fa$, then for all $0\leq k\leq n$,
$$\gamma ea[a_1,\ldots ,a_k]\;\in J_M(U)$$
For $k=0$, $[a_1,\ldots ,a_k]$ is $[\, ]$.

Of course, $\gamma ea[a_1,\ldots ,a_n]=e{\cdot}^fa$ since $e[a,a_1,\ldots ,a_n]=\pi {\sf T}(e{\cdot}^fa)$. Hence by assumption that $e{\cdot}^fa\in U$, we have $\gamma ea[a_1,\ldots ,a_n]=\alpha (e{\cdot}^fa)\in J_M(U)$.

Now suppose $k<n$ and $\gamma ea[a_1,\ldots ,a_{k+1}]\in J_M(U)$. Let $e[a,a_1,\ldots ,a_k]=\pi {\sf F}u_k$. We have
$$\gamma ea[a_1,\ldots ,a_k]=\zeta (\pi u_k\varepsilon )$$
where $\varepsilon =\langle x\rangle\gamma ea(\sigma\ast [x])$. Moreover, $f(u_k)=a_{k+1}$. We see that
$$\varepsilon f(u_k)=\gamma ea[a_1,\ldots ,a_{k+1}]\in J_M(U)$$
so $\pi u_k\varepsilon\in M(J_M(U))$, whence $\zeta (\pi u_k\varepsilon )\in J_M(U)$. This proves the claim.

We conclude that whenever $\pi ae\in J_f(U)$, that is $e{\cdot}^fa\in U$, we have $\gamma ea[\, ]\in J_M(U)$; so $J_f\leq J_M$ as desired.\end{proof} 
\medskip

\noindent As an example of a monomorphism of type 1, we consider the inclusion of assemblies $2\to\nabla (2)$. It is a result of Hyland, that the least local operator in $\Eff$ inverting this mono, is $\neg\neg$; we shall see whether this holds for arbitrary realizability toposes. The following lemma is from \cite{HylandJ:efft} and generalizes to arbitrary realizability toposes in a straightforward way.
\begin{lemma}\label{notnotchar} Let $J$ be a local operator. Then $\neg\neg\leq J$ if and only if the set $\bigcap_{a\in A}J(\{ a\} )$ is nonempty.\end{lemma}
We can represent the object $2$ in ${\sf RT}(A)$ as the assembly $(\{ 0,1\} ,E)$ with $E(0)=\{ \overline{0}\}$ and $E(1)=\{ \overline{1}\} $ (recall that $\overline{0},\overline{1}$ are the first two Curry numerals). Therefore the least local operator inverting $2\to\nabla (2)$ is $J_M$ where 
$$M(U)\; =\; (\{ \overline{0}\}\Rightarrow U)\cup (\{ \overline{1}\}\Rightarrow U)$$
Note, that $M$ is also the least monotone map with the property that $M(\{\overline{0}\} )\cap M(\{\overline{1}\} )$ is nonempty, and therefore $J_M$ is the least local operator $J$ for which $J(\{\overline{0}\} )\cap J(\{\overline{1}\} )$ is nonempty.
\begin{lemma}\label{2nabla2lemma} The least local operator which inverts the inclusion $2\to\nabla (2)$ is (up to isomorphism) the map $J$ which sends $U\subseteq A$ to $\bigcup_{n\in\mathbb{N}}(\{\overline{n}\}\Rightarrow U)$.\end{lemma}
\begin{proof} Martin Hyland showed in \cite{HylandJ:efft}, 16.4, that whenever $J$ is a local operator in $\Eff$ such that $J(\{ 0\} )\cap J(\{ 1\} )$ is nonempty, then $\bigcap_{n\in\mathbb{N}}J(\{ n\} )$ is nonempty. Since the tools for this proof were basic recursion theory, this proof generalizes to an arbitrary pca $A$ to yield: whenever $J$ is a local operator in ${\sf RT}(A)$ such that $J(\{\overline{0}\} )\cap J(\{\overline{1}\} )$ is nonempty, then $\bigcap_{n\in\mathbb{N}}J(\{\overline{n}\} )$ is nonempty.

Now the least monotone map $M$ such that $\bigcap_{n\in\mathbb{N}}M(\{\overline{n}\} )$ is nonempty, is the map $J$ in the statement of the lemma. So it remains to show that this is a local operator. Clearly, it is a monotone map, and certainly $\langle xy\rangle x$ is an element of $U\Rightarrow J(U)$ for all $U\subseteq A$. As to $J(J(U))\Rightarrow J(U)$, we note that we have uniform isomorphisms
$$\begin{array}{lll} J(J(U)) & \cong & \bigcup_n(\{ n\}\Rightarrow \bigcup_m(\{ m\}\Rightarrow U)) \\
 & \cong & \bigcup_{m,n}(\{ m\}\wedge\{ n\}\Rightarrow U) \\
  & \cong & \bigcup_k(\{ k\}\Rightarrow U)\; =\; J(U)\end{array}$$
The last isomorphism is because there exists a recursive pairing on the natural numbers which is a bijection from $\mathbb{N}\times\mathbb{N}$ to $\mathbb{N}$, and which is representable in $A$, as well as its unpairing functions.\end{proof}
\begin{theorem}\label{dmtheorem} For a pca $A$ the following three statements are equivalent:\begin{rlist}
\item The least local operator inverting $2\to\nabla (2)$ is $\neg\neg$.
\item There is an element $h\in A$ such that for every $a\in A$ there is a natural number $n$ satisfying $h\overline{n}=a$.
\item There exists a (necessarily essentially unique) geometric morphism ${\sf RT}(A)\to\Eff$.
\end{rlist}\end{theorem}
\begin{proof} This is now a triviality: given the characterizations of Lemma~\ref{notnotchar} and Lemma~\ref{2nabla2lemma}, we have the equivalence of i) and ii). But clearly, ii) is equivalent to the statement that the essentially unique decidable applicative morphism ${\cal K}_1\to A$, which is the map sending $n$ to $\overline{n}$, is computationaly dense. And that is equivalent to iii).\end{proof}
\begin{remark}\label{ptjremarks}\em We are grateful to Peter Johnstone for the following remark. As pointed out by Olivia Caramello in \cite{CaramelloO:demt}, the least local operator inverting $2\to\nabla (2)$, is also the least local operator for which the category of sheaves is a {\em De Morgan topos} (A topos is De Morgan if 2 is a $\neg\neg$-sheaf).

This yields another proof of iii)$\Rightarrow$i) in Theorem~\ref{dmtheorem}: if $f:{\sf RT}(A)\to {\sf RT}(B)$ is a geometric morphism, then $f$ restricts to a geometric morphism ${\sf RT}(A)_{\rm dm}\to {\sf RT}(B)_{\rm dm}$ (where ${\cal E}_{\rm dm}$ denotes the largest De Morgan subtopos of $\cal E$). This is immediate, because $f^*$ preserves both $2$ and $\nabla (2)$. This means that if ${\sf RT}(B)_{\rm dm}={\rm Set}$, then also ${\sf RT}(A)_{\rm dm}={\rm Set}$.\end{remark}
\begin{example}\label{effnumexamples}\em Peter Johnstone has suggested the terminology {\em effectively numerical\/} for a pca $A$ satisfying ii) of \ref{dmtheorem}. Clearly, if a pca is effectively numerical, it must be countable. The pca ${\cal K}_2^{\rm rec}$ is effectively numerical.

In order to see a countable pca which is nevertheless not effectively numerical, consider a nonstandard model of Peano Arithmetic $A$. $A$ is a pca if we define $ab=c$ to hold precisely if the formula $\exists x(T(a,b,x)\wedge U(x)=c)$ is true in $A$ (here $T$ and $U$ are Kleene's well-known computation predicate and output function; these things can be expressed in the language of Peano Arithmetic, hence interpreted in $A$). $A$ will then satisfy the axioms for a pca, since these are consequences of Peano Arithmetic. In $A$, the Curry numerals can be identified with the standard part of $A$. Now consider the following $\mathbb{N}$-indexed family of formulas in one variable $x$:
$$\Phi _a(x)\; =\;\{ \forall y(T(a,n,y)\to U(y)\neq x)\, |\, n\in\mathbb{N}\}$$
By \cite{KayeR:modpa}, Theorem~11.5, $A$ is {\em saturated\/} for types like this: there is an element $\xi\in A$ such that $\Phi _a(\xi )$ holds in $A$. That means, there is no $n$ such that $an=\xi$. Since $a$ is arbitrary, we see that $A$ cannot be effectively numerical.\end{example}
We conclude this paper with a characterization of those local operators in ${\sf RT}(A)$ for which the category of sheaves is ${\sf RT}(A[f])$ for some partial function $f$ on $A$. From Theorem~\ref{discthm} we know that if $\gamma :A\to B$ is discrete and projective and induces an inclusion, then this inclusion is of the form ${\sf RT}(A[f])\to {\sf RT}(A)$. Moreover, we know then that the local operator $J$ corresponding to this inclusion has the following properties:\begin{arlist}
\item $J(\{ a\} )\cap J(\{ b\} )=\emptyset$ whenever $a\neq b$ (we may call $J$ {\em discrete})
\item $J$ preserves unions.\end{arlist}
\begin{proposition} Suppose $J$ is a discrete local operator which preserves unions. Then there is a partial function $f$ on $A$ such that $J$ is isomorphic to $J_f$, the least local operator forcing $f$ to be realizable.\end{proposition}
\begin{proof} Define $f$ by: $f(a)=b$ if and only if $a\in J(\{ b\} )$. This is well-defined since $J$ is discrete. Let $M$ be the monotone map of the proof of Theorem~\ref{relrecA}, so $M(U)\; =\;\{ \pi ae\, |\, ef(a)\in U\}$ and $J_f$ is the least local operator majorizing $M$. Let $g$ realize the monotonicity of $J$: 
$$g\in\bigcap_{U,V\subseteq A}(U\Rightarrow V)\Rightarrow (JU\Rightarrow JV)$$
Now if $\pi ea\in M(U)$, then $e\in \{ f(a)\}\Rightarrow U$ so $ge\in J(\{ f(a)\} )\Rightarrow J(U)$, so $ge\in \{ a\}\Rightarrow J(U)$ (since $a\in J(\{ f(a)\} )$), so $gea\in J(U)$. This shows that $M\leq J$ and hence $J_f\leq J$. 

Conversely, if $a\in J(U)$ then since $J$ preserves unions, we have $a\in J(\{ x\} )$ for some $x\in U$, which means $f(a)\in U$, which implies that $\pi a{\sf i}$ (where $\sf i$ is such that ${\sf i}b=b$ for all $b\in A$) is an element of $M(U)$. So $J\leq M\leq J_f$. Note that we actually prove that $M$ is a local operator in this case!\end{proof}

\begin{small}
\bibliographystyle{plain}

\begin{thebibliography}{10}

\bibitem{CaramelloO:demt}
Olivia {C}aramello.
\newblock De {M}organ classifying toposes.
\newblock {\em Advances in {M}athematics}, 222(6):2117--2144, 2009.

\bibitem{CarboniA:somfcr}
A.~Carboni.
\newblock Some free constructions in realizability and proof theory.
\newblock {\em Journal of Pure and Applied Algebra}, 103:117--148, 1995.

\bibitem{FaberE:effott}
Eric {F}aber and {J}aap~van {O}osten.
\newblock Effective {O}perations of {T}ype 2 in {P}cas, 2014.
\newblock Submitted.

\bibitem{HofstraP:ordpca}
P.~Hofstra and J.~van Oosten.
\newblock Ordered partial combinatory algebras.
\newblock {\em Math. Proc. Camb. Phil. Soc.}, 134:445--463, 2003.

\bibitem{HylandJ:efft}
J.M.E. Hyland.
\newblock The effective topos.
\newblock In A.S. Troelstra and D.~Van Dalen, editors, {\em The {L.E.J.
  Brouwer} Centenary Symposium}, pages 165--216. North Holland Publishing
  Company, 1982.

\bibitem{JohnstonePT:skee}
P.T. Johnstone.
\newblock {\em {S}ketches of an {E}lephant (2 vols.)}, volume~43 of {\em Oxford
  Logic Guides}.
\newblock Clarendon Press, Oxford, 2002.

\bibitem{JohnstonePT:geomrt}
P.T. Johnstone.
\newblock Geometric {M}orphisms of {R}ealizability {T}oposes.
\newblock {\em Theory and {A}pplications of {C}ategories}, 28(9):241--249,
  2013.

\bibitem{JohnstonePT:glecrt}
P.T. Johnstone.
\newblock The {G}leason {C}over of a {R}ealizability {T}opos.
\newblock {\em Theory and {A}pplications of {C}ategories}, 28(32):1139--1152,
  2013.

\bibitem{JohnstonePT:irt}
P.T. Johnstone and E.P. Robinson.
\newblock A note on inequivalence of realizability toposes.
\newblock {\em Mathematical Proceedings of Cambridge Philosophical Society},
  105:1--3, 1989.

\bibitem{KayeR:modpa}
R.~Kaye.
\newblock {\em {M}odels of {P}eano {A}rithmetic}, volume~15 of {\em {O}xford
  {L}ogic {G}uides}.
\newblock {O}xford {U}niversity {P}ress, {O}xford, 1991.

\bibitem{LeeS:basse}
Sori {L}ee and {J}aap~van Oosten.
\newblock Basic {S}ubtoposes of the {E}ffective {T}opos.
\newblock {\em Annals of {P}ure and {A}pplied Logic}, 164:866?--883, 2013.

\bibitem{LongleyJ:reatls}
J.~Longley.
\newblock {\em Realizability Toposes and Language Semantics}.
\newblock PhD thesis, Edinburgh University, 1995.

\bibitem{PhoaW:relcet}
W.~Phoa.
\newblock Relative computability in the effective topos.
\newblock {\em Mathematical Proceedings of the Cambridge Philosophical
  Society}, 106:419--422, 1989.

\bibitem{PittsAM:thet}
A.M. Pitts.
\newblock {\em The Theory of Triposes}.
\newblock PhD thesis, Cambridge University, 1981.
\newblock available at {\tt
  http://www.cl.cam.ac.uk/$\sim$amp12/papers/thet/thet.pdf}.

\bibitem{OostenJ:extrea}
J.~van Oosten.
\newblock Extensional realizability.
\newblock {\em Annals of Pure and Applied Logic}, 84:317--349, 1997.

\bibitem{OostenJ:genfrr}
J.~van Oosten.
\newblock {A} general form of relative recursion.
\newblock {\em Notre Dame Journ. Formal Logic}, 47(3):311--318, 2006.

\bibitem{OostenJ:reaics}
J.~van Oosten.
\newblock {\em Realizability: an Introduction to its Categorical Side}, volume
  152 of {\em Studies in Logic}.
\newblock North-Holland, 2008.

\bibitem{OostenJ:parcaf}
J.~van Oosten.
\newblock {P}artial {C}ombinatory {A}lgebras of {F}unctions.
\newblock {\em Notre Dame Journ. Formal Logic}, 52(4):431--448, 2011.

\bibitem{OostenJ:realop}
J.~van {O}osten.
\newblock {R}ealizability with a {L}ocal {O}perator of {A}.{M}. {P}itts.
\newblock {\em {T}heoretical {C}omputer {S}cience}, 546:237?--243, 2014.
\newblock {A}vailable at {\tt http://dx.doi.org/10.1016/j.tcs.2014.03.011}.
\end{thebibliography}

\end{small}
\end{document}